\documentclass{birkjour}
\usepackage{mathrsfs}
\newtheorem{theorem}{Theorem}[section]

\newtheorem{corollary}{Corollary}[section]

\renewcommand{\theequation}{\thesection.\arabic{equation}}
\renewcommand{\theequation}{\thesection.\arabic{equation}}

\begin{document}

\title[Gaussian lower bound]{A remark on the Gaussian lower bound for the Neumann heat kernel of the Laplace-Beltrami operator}

\author[Mourad Choulli]{Mourad Choulli}
\address{Institut \'Elie Cartan de Lorraine, UMR CNRS 7502, Universit\'e de Lorraine, Boulevard des Aiguillettes, BP 70239, 54506 Vandoeuvre les Nancy cedex - Ile du Saulcy, 57045 Metz cedex 01, France}
\email{mourad.choulli@univ-lorraine.fr}
\author[Laurent Kayser]{Laurent Kayser}
\address{Institut \'Elie Cartan de Lorraine, UMR CNRS 7502, Universit\'e de Lorraine, Boulevard des Aiguillettes, BP 70239, 54506 Vandoeuvre les Nancy cedex - Ile du Saulcy, 57045 Metz cedex 01, France}
\email{laurent.kayser@univ-lorraine.fr}

\date{}

\begin{abstract}
We adapt in the present  note the perturbation method introduced in \cite{CK} to get a lower Gaussian bound for the Neumann heat kernel of the Laplace-Beltrami operator on an open subset of a compact Riemannian manifold. 
 
 
\end{abstract}

\maketitle


\section{Introduction}

The study of heat kernels  is an important problem in the theory of parabolic PDE's. The properties of heat kernels give an efficient tool to answer to some central questions  both in analysis and probability theory. One of the main questions is to know whether a heat kernel  admits Gaussian bounds. An upper Gaussian bound is for instance an useful tool for getting $L^p$-$L^q$ estimates, the analyticity of the corresponding semigroups in $L^p$ for any finite $p\geq1$ or bounded functional calculus, whereas one can get a strong maximum principle or a Harnack inequality from a  lower Gaussian bound. We refer to the textbooks \cite{Dav}, \cite{Ou} and \cite{St} and references therein for more details on the subject.

\smallskip
In the preceding work \cite{CK}, starting from the classical  parametrix method, we constructed the Neumann heat kernel of a general parabolic operator as a perturbation of the fundamental solution of the same operator by a single-layer potential. From this construction, the two-sided Gaussian bounds for the fundamental solution and taking into account the smoothing effect in time of the single-layer potential, we succeeded in proving a lower Gaussian bound for the Neumann Green function. We adapt in the present note this method  to establish a lower Gaussian bound for the Neumann heat kernel of Laplace-Beltrami operator.

\smallskip
In this text $\mathcal{M}=(\mathcal{M},g)$ is a $n$-dimensional compact connected Riemannian manifold without boundary and $\Omega$ is a  domain in $\mathcal{M}$ so that its boundary $\Sigma$ is an $(n-1)$-dimensional Riemannian submanifold of $M$ when it is equipped with the metric induced by $g$.

\smallskip
The Riemannian measure on $\mathcal{M}$ is denoted by $dV$ while the density measure on $\Sigma$ is denoted by $dA$. The geodesic ball of center $x\in \mathcal{M}$ and radius $r>0$ is denoted by $B(x,r)$.

\section{Neumann heat kernel}

\smallskip
Let $d$ be the Riemannian distance function and
\[
\mathscr{E}(x,y,t)=(4\pi t)^{-n/2}e^{-\frac{d^2(x,y)}{4t}}.
\]
In general $\mathscr{E}$ is not a heat kernel on $\mathcal{M}$. However the parametrix method by S. Minakshisundaram and \AA.Pleijel shows that any compact Riemannian manifold has an almost euclidien heat kernel (e.g. \cite{BGM, Cha}). In particular, the heat kernel $p$ of $\mathcal{M}$, satisfies 
\[
p(x,y,t)\sim \mathscr{E}(x,y,t),
\]
locally uniformly in $(x,y)$ as $t\downarrow 0$. We have a similar statement with the first order derivatives of $p$ and $\mathscr{E}$.

\smallskip
This estimate is not true in general for distant $x$ and $y$. Indeed the counter example in \cite[Example 3.1, page 23]{Mo} shows that if $\mathscr{M}=\mathbb{S}^2$, the 2-dimensional sphere equipped with the round metric, $x$ is the north pole and $y$ is the south pole, then
\[
p(x,y,t)\sim ct^{-3/2}e^{-\frac{d^2(x,y)}{4t}},
\]
for some constant $c>0$.

\smallskip
From \cite[Theorem 5.5.11 and Theorem 5.6.1, page 173]{Dav}, any complete Riemannian manifold with non negative Ricci curvature satisfies the following two-sided Gaussian bounds.
\[
\mathscr{E}(x,y,t) \le p(x,y,t)\le \frac{c}{V(B(x,\sqrt{t}))}e^{-\kappa \frac{d^2(x,y)}{4t}},\;\; x,y\in \mathcal{M},\; t>0.
\]
Here $c>0$ and $\kappa >0$ are some constants.

\medskip
Let $\Delta =\Delta _g$ be the Laplace-Beltrami operator associated to the metric $g$ and denote by $\nu$ the outward normal vector field to $\Sigma$. Following the idea in \cite{CK}, we construct the Green function of  the Neumann problem
\begin{equation}\label{e1}
\left\{
\begin{array}{ll}
(\partial _t-\Delta )u=0\;\; &\mbox{in}\; \Omega \times (0,+\infty ),
\\
\frac{\partial u}{\partial \nu}=0, &\mbox{on}\; \Sigma \times (0,+\infty ),
\end{array}
\right.
\end{equation}
as a perturbation of the heat kernel $p$ by a single-layer potential. As a first step, we seek the solution, in $C^{2,1}(\Omega \times (0,+\infty ))\cap C^{0,1}(\overline{\Omega}\times [0,+\infty ))$, of the following IBVP
\begin{equation}\label{e2}
\left\{
\begin{array}{ll}
(\partial _t-\Delta )u=0\;\; &\mbox{in}\; \Omega \times (0,+\infty ),
\\
\frac{\partial u}{\partial \nu}=0, &\mbox{on}\; \Sigma \times (0,+\infty ),
\\
u(\cdot ,0)=\psi \in C_0^\infty (\Omega)
\end{array}
\right.
\end{equation}
of the form 
\begin{align*}
u(x,t)=\int_0^t\int_\Sigma p(x,y,s)\varphi (y,t-s)dA(y)ds+\int_\Omega p(x,&y,t)\psi (y)dV(y),
\\ &(x,t)\in \Omega \times (0,+\infty ).
\end{align*}
We obtain from the jump relation in \cite[Theorem 2, page 161]{Cha} that $\varphi$ must be the solution of the following integral equation
\begin{align*}
\varphi (x,t)=-2\int_0^t\int_\Sigma \frac{\partial p}{\partial \nu _x}&(x,y,s)\varphi (y,t-s)dA(y)ds
\\
&-2\int_\Omega \frac{\partial p}{\partial \nu _x}(x,y,t)\psi (y)dV(y),\;\; (x,t)\in \Sigma \times (0,+\infty ).
\end{align*}
This integral equation is solved by successive approximations. We get
\[
\varphi (x,t)=\varphi _0(x,t)+\int_0^t\int_\Sigma r(x,y,s)\varphi _0(y,t-s)dA(y)ds,\;\; (x,t)\in \Sigma \times (0,+\infty ).
\]
Here
\begin{align*}
&\varphi _0(x,t)=-2\int_\Omega \frac{\partial p}{\partial \nu _x}(x,y,t)\psi (y)dV(y),\;\; (x,t)\in \Sigma \times (0,+\infty ),
\\
&r(x,y,t)=\sum_{j\geq 1} r^j(x,y,t),\;\; (x,t)\in \Sigma \times (0,+\infty ),\; y\in \overline{\Omega},
\end{align*}
with
\begin{align*}
&r^1(x,y,t)=-2\frac{\partial p}{\partial \nu _x}(x,y,t),\;\; (x,t)\in \Sigma \times (0,+\infty ),\; y\in \overline{\Omega},
\\
&r^{j+1}(x,y,t)=-2\int_0^t\int_\Sigma r^1(x,z,t-s)r^j(z,y,s)dA(z)ds,\;\; j\geq 1,
\\
&\hskip 7cm (x,t)\in \Sigma \times (0,+\infty ),\; y\in \overline{\Omega}.
\end{align*}

In other words
\begin{align*}
r(x,y,t)=-2\frac{\partial p}{\partial \nu _x}(x,y,t) -2\sum_{j\geq 1}\int_0^t\int_\Sigma r^j(x,&z,s)\frac{\partial p}{\partial \nu _x}(z,y,t-s)dA(z)ds
\\ &(x,t)\in \Sigma \times (0,+\infty ), \; y\in \overline{\Omega}.
\end{align*}

The following two inequalities will be useful in the sequel. They are taken from \cite{Cha}.

\smallskip
For any $\mu >0$, there is a constant $C_0$ such that
\begin{equation}\label{e3}
\left| \frac{\partial p}{\partial \nu _x}(x,y,t) \right|\leq C_0t^{-\mu}d^{-n+2\mu}(x,y),
\end{equation}
for any $x\in \Sigma$, $y\in \overline{\Omega}$ and $t\in (0,T]$.

\smallskip
There exists a constant $C_1$ such that, for any $\alpha ,\beta \in (0,n-1)$,
\begin{equation}\label{e4}
\int_\Sigma d^{-\alpha}(x,z)d^{-\beta}(z,y)dA(z)\leq C_1\left\{ \begin{array}{ll} d^{n-1-(\alpha +\beta )}(x,y)&\;\; \textrm{if}\; \alpha +\beta > n-1,\\ 1 &\;\; \textrm{if}\; \alpha +\beta < n-1.\end{array}\right.
\end{equation}
for all $x,y\in \overline{\Omega}$, $x\neq y$.

With the help of inequality \eqref{e3}, we prove similarly to \cite[(3.6)]{CK}
\[
\varphi (x,t)=\int_\Omega r(x,y,t)\psi (y)dV(y).
\]
Then
\[
u(x,t)=\int_\Omega q(x,y,t)\psi (y) dV(y),
\]
where
\[
q(x,y,t)=p(x,y,t)+\int_0^t\int_\Sigma p(x,z,s)r(z,y,t-s)dA(z)ds.
\]
We call this function the Neumann heat kernel for the problem \eqref{e1}. We leave to the reader to verify that $q$ satisfies the following reproducing property
\[
q(x,y,t)=\int_\Omega q(x,z,t-s)q(z,y,s)dV(z),\;\; x,y\in \Omega ,\; 0<s<t,
\]
and
\[
\int_\Omega q(x,y,t)dV(y)=1,\;\; x\in \Omega ,\; t>0.
\]


\section{Gaussian lower bound}

We set $v(x,r)= V(B(x,r)\cap \Omega )$, $x\in \Omega$, and we consider the following three assumptions.

\smallskip
$(VLB)$ (volume lower bound)  There exist two constants $C$ and $r_0$ so that 
\[
v(x,r)\geq Cr^n,\;\; x\in \Omega ,\; 0<r\leq r_0.
\]

\smallskip
$(DP)$ (doubling property) There exist two constants $r_1>0$ and $C>0$ so that \[ v(x,s)\leq C\left(\frac{s}{r}\right)^nv(x,r), \] for all $0<r\leq s\leq r_1$ and $x\in \Omega$.

\smallskip
$(CC)$ (chain condition) There exists a constant $C>0$ such that for any $x,y\in \Omega$ and $k\in \mathbb{N}$, we find a sequence of points $(x_i)_{0\leq i\leq k}$ so that $x_0=x$, $x_k=y$ and
\[
d(x_i,x_{i+1})\leq C\frac{d(x,y)}{k},\;\; 0\leq i\leq k-1.
\]

In the flat case, $(VLC)$ and $(DP)$ are true for any Lipschitz domain  and $(CC)$ holds for instance for a convex domain.

\smallskip
Let $\mathbb{S}^1$ be the unit sphere of $\mathbb{R}^2$ equipped with the round metric, that is the metric induced by the Euclidean metric on $\mathbb{R}^2$. Then any segment of length strictly less than $\pi$ possesses the the three conditions $(VLB)$, $(DP)$ and $(CC)$. Let $\mathbb{S}^2$ be the unit sphere of $\mathbb{R}^3$ equipped with the round metric. Then it is not hard to show that the sub-domain of $\mathbb{S}^2$ given by $z>\delta >0$ satisfies also the three conditions $(VLB)$, $(DP)$ and $(CC)$.

\smallskip 
Conditions $(DP)$ and $(CC)$ are usual (see for instance \cite[Theorem 7.29, page 248]{Ou}), while condition $(VLB)$ guarantees that the near diagonal lower bound \eqref{e8} above holds. 

\smallskip
We aim to sketch the proof of the following theorem

\begin{theorem}\label{thmGLB}
Fix $T>0$ and assume that $\Omega$ satisfies $(VLB)$, $(DP)$ and $(CC)$. Then
\begin{equation}\label{e9}
q(x,y,t)\geq \frac{c}{v(x,\sqrt{t})}e^{-\frac{d^2(x,y)}{ct}}, \;\; x,y\in \Omega ,\; 0<t\leq T.
\end{equation}
\end{theorem}

\begin{proof}[Sketch of the proof]
Let $1/2 <\mu < n/2$. In light of \eqref{e3} and \eqref{e4}, reasoning as in the proof of \cite[Lemma 3.1]{CK}, we obtain
\[
|r(x,y,t)|\leq Ct^{-\mu}d^{-n+2\mu}(x,y),
\]
for any $x\in \Sigma $, $y\in \overline{\Omega}$, $x\neq y$, $t\in (0,T]$.

Let
\[
q_0(x,y,t)=\int_0^t\int_\Sigma p(x,z,s)r(z,y,t-s)dA(z)ds,\;\; x,y\in \Omega ,\; t>0,
\]
and $0<\alpha <1/2$. We proceed as in the beginning of the proof of \cite[Theorem 3.1]{CK} to get
\begin{equation}\label{e5}
|q_0(x,y,t)|\leq Ct^{-n/2+\alpha }\;\; x,y\in \Omega ,\; t>0.
\end{equation}

Let $\epsilon :=\textrm{inj}(\mathcal{M})/4$, where $\textrm{inj}(\mathcal{M})$ is the injectivity radius of $\mathcal{M}$. It follows from \cite[formula (45), page 154]{Cha} that there exists $\eta >0$ so that
\begin{equation}\label{e6}
p(x,y,t)\geq \mathscr{E}(x,y,t),\;\;  0<t\le \eta ,\; x,y\in \mathcal{M},\; d(x,y)\leq \epsilon .
\end{equation}
Hence,
\begin{equation}\label{e7}
p(x,y,t)\geq Ct^{-n/2},\;\; 0<t\le \inf (\eta ,\epsilon ^2),\; x,y\in \mathcal{M},\; d(x,y)\leq \sqrt{t}.
\end{equation}

Now a combination of \eqref{e5} and \eqref{e7} leads
\begin{align*}
q(x,y,t)\geq p(x,y,t)-|q_0(x,y,y)|&\geq Ct^{-n/2}(1-ct^\alpha),
\\
& 0<t\le \inf (\eta ,\epsilon ^2),\; x,y\in \Omega ,\; d(x,y)\leq \sqrt{t}.
\end{align*}
In consequence, there is $\delta >0$ such that
\begin{equation}\label{e8_0}
q(x,y,t)\geq Ct^{-n/2},\;\; \mbox{if}\;\; 0<t\leq \delta ,\;\; x,y\in \Omega ,\;\; d(x,y)\leq \sqrt{t}.
\end{equation}
In light of the volume lower bound $(VLB)$, this estimate entails
\begin{equation}\label{e8}
q(x,y,t)\geq \frac{C}{v(x,\sqrt{t})},\;\; \mbox{if}\;\; 0<t\leq \widetilde{\delta} ,\;\; x,y\in \Omega\;\; d(x,y)\leq \sqrt{t},
\end{equation}
for some constant $\widetilde{\delta}$.

\smallskip
We can now mimic the proof of \cite[Theorem 7.29, page 248]{Ou}. We get from \eqref{e8} the following Gaussian lower bound
\[
q(x,y,t)\geq \frac{c}{v(x,\sqrt{t})}e^{-\frac{d^2(x,y)}{ct}}, \;\; x,y\in \Omega ,\; 0<t\leq \widetilde{\delta}.
\]
We finally use the argument as in \cite[Theorem 3.1]{CK} to pass from $0<t\leq \widetilde{\delta}$ to $0<t\leq T$.
\end{proof}

It is worthwhile mentioning that one can establish a lower Gaussian bound by considering $\mathcal{N}=\overline{\Omega}$ itself as a compact Riemannian manifold with boundary. The structure of Riemannian manifold is the one inherited from $\mathcal{M}$. Obviously, \eqref{e8_0} entails
\begin{equation}\label{e11}
q(x,y,t)\geq Ct^{-n/2},\;\;  0<t\leq \delta ,\;\; x, y\in \mathcal{N},\;\; d_\mathcal{N} (x,y)\leq \sqrt{t}.
\end{equation}
Assume that the Ricci curvature of $\mathcal{N}$ is such that $Ric\geq (n-1)\kappa g$ for some $\kappa \in \mathbb{R}$. Then $\mathcal{N}$ satisfies the a doubling property $(DP)$ when $V$ is substituted by $V_\mathcal{N}$, the volume measure over $\mathcal{N}$. This fact is an immediate consequence of \cite[formula in the bottom of page 7]{He}. Let $v_\mathcal{N}(x,r)=V_\mathcal{N}(B(x,r))$, where $B(x,r)$ is the geodesic ball in $\mathcal{N}$ of center $x\in \mathcal{N}$ and radius $r$. In that case we can proceed as in the proof of Theorem \ref{thmGLB} to derive the following estimate.
\[
q(x,y,t)\geq \frac{c}{v_\mathcal{N}(x,\sqrt{t})}e^{-\frac{d_\mathcal{N}^2(x,y)}{ct}}, \;\; x,y\in \Omega ,\; 0<t\leq T.
\]
This estimate should be compared to the one obtained by Li and Yau in \cite[Theorem 4.2, page 184]{LY}. Specifically, they established a lower Gaussian bound for the heat kernel of a compact Riemannian manifold with convex boundary and having non negative Ricci curvature.

\smallskip
When $\mathcal{M}$ is any complete Riemannian manifold with finite diameter and having volume doubling property, and $\Omega$ is Lipschitz domain in $\mathcal{M}$ with volume doubling property, the Neumann heat kernel of $\Omega$, denoted here by $h$, satisfies the following upper Gaussian bound.
\[
h(x,y,t)\le \frac{C}{v (x,\sqrt{t})v(y,\sqrt{t})}e^{-\frac{d^2(x,y)}{8t}},\;\; x,y\in \Omega ,t>0,
\]
where $C>0$ is some constant.

\smallskip
This estimate was recently established by the authors and E. M. Ouhabaz \cite{CKO}.


\section{Comments on geometric assumptions} {\it Chain condition}: A subset $\mathcal{C}$ of $\mathcal{M}$ is called strongly convex if for any $x,y\in \mathcal{C}$, there exists a unique minimal geodesic $\gamma :[0,1]\rightarrow \mathcal{M}$ joining $x$ to $y$, so that $\gamma ([0,1])\subset \mathcal{C}$. According to a theorem due to Whitehead (see for instance \cite[pages 161 and 162]{GKM}), there exists a positive continuous function $\epsilon :\mathcal{M}\rightarrow (0,\infty ]$, the convexity radius, such that any open ball $B(x,r)\subset B(x,\epsilon (x))$ is strongly convex. It is straightforward to check that if $\Omega$ is strongly convex then it has the chain condition.

\smallskip
{\it Volume lower bound}: Let $T_x\mathcal{M}$ be the tangent space at $x\in \mathcal{M}$, $\mathbb{S}_x\subset T_x\mathcal{M}$ the unit tangent sphere and $S\mathcal{M}$ the unit tangent bundle. Let $\Phi_t$ be the geodesic flow with phase space $S\mathcal{M}$. That is, for any $t\geq 0$, 
\[
\Phi_t:S\mathcal{M}\rightarrow S\mathcal{M}: (x,\xi )\in S\mathcal{M} \rightarrow\Phi_t (x,\xi )=\left( \gamma_{x,\xi }(t),\dot{\gamma}_{x,\xi }(t)\right).
\]
Here $\gamma_{x,\xi}:[0,\infty )\rightarrow \mathcal{M}$ is the unit speed geodesic starting at $x$ with tangent unit vector $\xi$ and $\dot{\gamma}_{x,\xi}(t)$ is the unit tangent vector to $\gamma_{x,\xi}$ at $\gamma_{x,\xi}(t)$ in the forward $t$ direction.

\smallskip
If $(x,\xi)\in S\mathcal{M}$, we denote by $r(x,\xi)$ the distance from $x$ to the cutlocus in the direction of $\xi$:
\[
r(x,\xi )=\inf\{t>0;\; d(x,\Phi_t(x,\xi ))<t\}.
\]

We fix $\delta \in (0,1]$ and $r>0$.  Following \cite{Sa}, a $(\delta ,r)$-cone at $x\in \mathcal{M}$ is the set of the form
\[
\mathscr{C}(x,\omega _x,r)=\{ y=\gamma_{x,\xi}(s);\; \xi\in \omega _x,\; 0\leq s<r\},
\]
where $\omega _x$ is a subset of $\mathbb{S}_x$ so that $r<r(x,\xi )$ for all $\xi \in \omega_x$ and $|\omega _x |\geq \delta$ (here $|\omega _x|$ is the volume of $\omega _x$ with respect to the normalized measure on the sphere $\mathbb{S}_x$).

\smallskip
A domain $D$ which contains an $(\delta ,r)$-cone at $x$ for any $x \in D$ is said to satisfy
the interior $(\delta ,r)$-cone condition.

\smallskip
We observe that if $\mathcal{C}$ is a closed strongly convex subset of $\mathcal{M}$, then $\Omega =\mathcal{M}\setminus \mathcal{C}$ has the $(1/2 ,r)$-cone condition, for some $r$ (this fact follows from the same argument to that in \cite[Example 8.1, page 370]{Sa}).

\smallskip
Let
\[
s_\kappa (r)=\left\{
\begin{array}{ll}
\left(\frac{\sin (\sqrt{\kappa}r}{\sqrt{\kappa}}\right)^{n-1}\;\; &\textrm{if}\; \kappa >0,
\\
r^{n-1} &\textrm{if}\; \kappa =0,
\\
\left(\frac{\sinh (\sqrt{-\kappa}r}{\sqrt{-\kappa}}\right)^{n-1}\  &\textrm{if}\; \kappa <0.
\end{array}
\right.
\]

\smallskip
We make the assumption that the sectional curvature of $\mathcal{M}$ is is bounded above by a constant $\kappa$, $\kappa \in \mathbb{R}$,  and $\Omega$  satisfies the interior $(\delta ,r)$-cone condition. Let $J(x,\xi ,t)$ be the density of the volume element in geodesic coordinates around $x$:
\[
dV(y)=J(x,\xi ,t)d_{\mathbb{S}_x}dt,\;\; y=\gamma _{x,\xi}(t),\; t<r(x,\xi ).
\]
By an extension of G\"unther's comparison theorem (see for instance \cite{KK}),  $J$ satisfies the following uniform lower bound
\[
J(x,\xi ,t)\geq s_\kappa (t).
\]
Consequently, shrinking $r_0$ if necessary, we have
\begin{equation}\label{ce}
v(x,r)\geq V (\mathscr{C}(x,\omega _x,r))\geq c_0 r^n,\;\; x\in \Omega ,\; 0<r\leq r_0 ,
\end{equation}
which means that $v$ satisfy the volume lower bound $(VLB)$.

\smallskip
Additionally, if  $\mathcal{M}$ satisfies the following volume growth condition 
\begin{equation}\label{vgc}
V(x,r)\leq c_1r^n, \;\; 0<r\leq r_0,
\end{equation}
for some constants $c_1$ and $r_1$ then $v$ has the doubling property $(DP)$.

\smallskip
As a consequence of Theorem \ref{thmGLB}, we have
\begin{corollary}\label{thmGLB2}
Assume that the sectional curvature of $\mathcal{M}$ is bounded from above,  the volume growth condition \eqref{vgc} is fulfilled and $\Omega$ is strongly convex and satisfies the interior $(\delta ,r)$-cone condition. Then
\begin{equation}\label{e10}
q(x,y,t)\geq c\mathscr{E}(x,y,ct), \;\; x,y\in \Omega ,\; 0<t\leq T.
\end{equation}
\end{corollary}


\medskip
\medskip

\end{document}